\documentclass[11pt,a4paper]{amsart}
\usepackage{amsmath,amssymb,amsthm}
\usepackage[margin=2.6cm]{geometry}
\usepackage{booktabs}
\usepackage{hyperref}
\usepackage{longtable}

\newtheorem{theorem}{Theorem}[section]

\newtheorem{proposition}[theorem]{Proposition}

\theoremstyle{definition}

\newtheorem{remark}[theorem]{Remark}

\newcommand{\C}{\mathbb{C}}
\newcommand{\Der}{\operatorname{Der}}
\newcommand{\Inn}{\operatorname{Inn}}
\newcommand{\QDer}{\operatorname{QDer}}
\newcommand{\GDer}{\operatorname{GDer}}
\newcommand{\TDer}{\operatorname{TDer}}
\newcommand{\Aut}{\operatorname{Aut}}
\newcommand{\End}{\operatorname{End}}
\newcommand{\Hom}{\operatorname{Hom}}
\newcommand{\Ann}{\operatorname{Ann}}
\newcommand{\Leib}{\operatorname{Leib}}
\newcommand{\id}{\mathrm{id}}
\newcommand{\HL}{\mathrm{HL}}
\newcommand{\n}{\mathfrak{n}}

\title[Derivations, automorphisms and deformations of a Leibniz algebra]{An eight-dimensional nilpotent Leibniz algebra without outer derivations}

\author{Kobiljon Abdurasulov}
\address{[Kobiljon Abdurasulov] University of Beira Interior, Rua Marquês D'Ávila e Bolama, 6201-001, Covilhã, Portugal}
\email{abdurasulov0505@mail.ru}

\author{Jobir Adashev}
\address{[Jobir Adashev] Institute of Mathematics, Uzbekistan Academy of Sciences, 9 University Street, 100174, Tashkent, Uzbekistan}
\email{adashevjq@mail.ru, adashevjq@gmail.com}
\date{}

\begin{document}

\begin{abstract}
Jacobson's classical theorem asserts that every nilpotent Lie algebra admits a non-inner derivation. This result, however, does not extend to Leibniz algebras: there exists an eight-dimensional complex nilpotent Leibniz algebra whose every derivation is inner. In this paper, we present a systematic investigation of this algebra, with particular emphasis on the absence of outer derivations and its structural properties. We describe its lower central series, annihilators, center, and Leibniz kernel, and prove that its derivation algebra is a six-dimensional nilpotent Lie algebra isomorphic to the quotient of the algebra by its Leibniz kernel. We further show that its automorphism group is a connected six-dimensional unipotent group obtained by exponentiating inner derivations; in particular, it contains no nontrivial elements of finite order. In addition, we provide a complete classification of the quasiderivation, generalized derivation, and $\sigma$-derivation of the algebra. Finally, we show that, among the four linearly independent infinitesimal deformations, exactly one is integrable, giving rise to a one-parameter family of pairwise non-isomorphic algebras which, with a single exception, also admit only inner derivations.
\end{abstract}


\maketitle

\setcounter{tocdepth}{2}
\tableofcontents

\section{Introduction}

Leibniz algebras were introduced by Bloh~\cite{Bloh} under the name of $D$-algebras and also by Loday~\cite{Loday} as the natural non-antisymmetric analogues of Lie algebras. A (right) Leibniz algebra is a vector space with a bilinear product for which every right multiplication is a derivation; a Lie algebra is precisely a Leibniz algebra in which the square of every element vanishes. Since the work of Loday and Pirashvili~\cite{LodayPirashvili} on universal enveloping algebras and (co)homology, the theory has developed in many directions: analogues of classical structure theorems such as Levi's theorem have been established~\cite{Barnes}, and the cohomology and deformation theory of Leibniz algebras has been developed~\cite{FialowskiMandalMukherjee}. A recurring theme is the comparison with Lie algebras: which Lie-theoretic results survive the loss of antisymmetry, and which fail? The present paper is a contribution to this comparison.

In a Leibniz algebra, the right multiplications $R_x\colon y\mapsto yx$ are derivations by definition; they are called \emph{inner}, and the quotient of all derivations by the inner ones is the first cohomology of the algebra with coefficients in itself. Jacobson proved~\cite{Jacobson} that every nilpotent Lie algebra of positive dimension possesses an outer derivation. This holds even for the \emph{characteristically nilpotent} Lie algebras of Dixmier and Lister~\cite{DixmierLister}, all of whose derivations are nilpotent. For Leibniz algebras the analogue of Jacobson's theorem fails: there is an explicit eight-dimensional nilpotent right Leibniz algebra $L$ whose derivations are all inner; moreover $\Der(L)=\Inn(L)$ is six-dimensional, $L$ is characteristically nilpotent.

The aim of this paper is to study an algebra usually understood through its automorphisms, its derivations and their relaxations (quasiderivations, generalized derivations, twisted derivations) and its deformations. Each of these is related to derivations in a specific way, and the guiding question of this paper is: \emph{how does the rigidity expressed by the absence of outer derivations propagate through these families?} All of them are actively studied for Lie and Leibniz algebras. Generalized derivations and quasiderivations go back to Leger and Luks~\cite{LegerLuks}, the related ternary derivations were studied by Jim\'enez-Gestal and P\'erez-Izquierdo~\cite{JGPI}, and $\sigma$-derivations underlie the quasi-Lie and Hom-Lie deformations of Hartwig, Larsson and Silvestrov~\cite{HLS}. Deformations of algebras were introduced by Gerstenhaber~\cite{Gerstenhaber} and of Lie algebras by Nijenhuis and Richardson~\cite{NR}; the Leibniz case was developed by Fialowski, Mandal and Mukherjee~\cite{FialowskiMandalMukherjee}.

\textbf{Main results.} Throughout, $L$ is the complex Leibniz algebra with basis $e_1,\dots,e_8$ and multiplication table~\eqref{eq:table}. In Section~\ref{sec:structure} show that $L$ is generated by two elements and is nilpotent of class $5$; its right annihilator coincides with its Leibniz kernel $\langle e_7,e_8\rangle$, its centre and left annihilator equal $\langle e_8\rangle$, and the quotient by the Leibniz kernel is the six-dimensional nilpotent Lie algebra $\n_6$. A one-parameter family $L_t$ of deformations of $L$ consists, for $t\ne-1$, of pairwise non-isomorphic nilpotent Leibniz algebras without outer derivations (Theorem~\ref{thm:Lt}), so $L$ is not an isolated example. In sections~\ref{sec:der} and~\ref{sec:aut} proved that $\Der(L)=\Inn(L)\cong\n_6$, and $\Aut(L)=\{\exp R_a: a\in L\}$ is a six-dimensional connected unipotent group isomorphic to the simply connected nilpotent Lie group with Lie algebra $\n_6$. Every automorphism induces the identity on $L/L^2$, and there are no non-trivial automorphisms of finite order. In Sections~\ref{sec:quasi} and~\ref{sec:sigma} sections, we classify quasiderivations and generalized derivations, obtaining spaces of dimensions $13$ and $22$, respectively. Moreover, all $\sigma$-derivations are characterized as maps of the form $\sigma\circ R_a$. Finally, in section~\ref{sec:def}, we show that among the four-dimensional space of infinitesimal deformations, a two-dimensional subspace is unobstructed at second order; one direction is obstructed at order four, while the other integrates to the family $L_t$. Hence, $L$ is not rigid.

\section{Preliminaries}

We fix notation and recall the definitions concerning Leibniz algebras, operators on algebras, cohomology and deformations that will be used in this paper. All vector spaces are over the field $\C$ of complex numbers unless stated otherwise.

A \emph{right Leibniz algebra} is a vector space $A$ with a bilinear product $(x,y)\mapsto xy$ satisfying
\begin{equation}\label{eq:leibniz}
(xy)z=x(yz)+(xz)y\qquad(x,y,z\in A).
\end{equation}
For $x\in A$ we write $R_x(y)=yx$ and $L_x(y)=xy$. Identity~\eqref{eq:leibniz} says that every $R_x$ is a derivation. The \emph{Leibniz kernel} is $\Leib(A)=\operatorname{span}\{xx:x\in A\}$; it is an ideal, $A/\Leib(A)$ is a Lie algebra, and $A$ is a Lie algebra if and only if $\Leib(A)=0$. The right and left annihilators are $\Ann_r(A)=\{x: Ax=0\}$ and $\Ann_l(A)=\{x:xA=0\}$, and the centre is $Z(A)=\Ann_r(A)\cap\Ann_l(A)$. A subspace $I$ is an \emph{ideal} if $AI+IA\subseteq I$.

The lower central series is $A^1=A$, $A^{k+1}=A^kA$; $A$ is \emph{nilpotent} if $A^{s}=0$ for some $s$, and its nilpotency class is the least $c$ with $A^{c+1}=0$.

Now we define the following classes of linear operators. They all arise by modifying the derivation identity $f(xy)=f(x)y+xf(y)$ in one of several ways: by allowing different operators in its three positions (generalized and quasiderivations) or by twisting with an endomorphism ($\sigma$-derivations). Let $A$ be an algebra and $f,f',f''\in\End(A)$.
\begin{itemize}
\item $f$ is a \emph{derivation} if $f(xy)=f(x)y+xf(y)$; $\Der(A)$ is a Lie algebra under the commutator. A nilpotent Leibniz algebra is \emph{characteristically nilpotent} if all its derivations are nilpotent.
\item $(f'',f,f')$ is a \emph{ternary derivation} if $f''(xy)=f(x)y+xf'(y)$ for all $x,y$; we denote the space of ternary derivations by $\TDer(A)$~\cite{JGPI}. 
\item $f$ is a \emph{generalized derivation} if there are $f',f''$ with $(f'',f,f')\in\TDer(A)$, and a \emph{quasiderivation} if there is $f''$ with $(f'',f,f)\in\TDer(A)$~\cite{LegerLuks}. We write $\GDer(A)$ and $\QDer(A)$; then $\Der(A)\subseteq\QDer(A)\subseteq\GDer(A)$.
\item For $\sigma\in\End(A)$, a \emph{$\sigma$-derivation} is a linear map $d$ with $d(xy)=d(x)\sigma(y)+\sigma(x)d(y)$~\cite{HLS}.
\end{itemize}

\textbf{Cohomology and deformations.} Deformation theory describes the algebras which are ``infinitesimally close'' to a given one. Its basic principle, due to Gerstenhaber~\cite{Gerstenhaber}, is that first-order deformations are classified by the second cohomology with adjoint coefficients, and that the extension of a deformation to higher orders is controlled by obstructions in the third cohomology. We recall the relevant definitions in the Leibniz setting.
For a right Leibniz algebra $(A,\mu)$, $\mu(x,y)=xy$, and bilinear maps $\varphi,\psi\colon A\times A\to A$ define the trilinear map
\[
(\varphi\circ\psi)(x,y,z)=\varphi(\psi(x,y),z)-\varphi(x,\psi(y,z))-\varphi(\psi(x,z),y).
\]
Thus $\mu$ is a Leibniz product if and only if $\mu\circ\mu=0$. The second Leibniz coboundary of a $2$-cochain $\psi$ (with adjoint coefficients) is, up to sign, $\delta\psi=\mu\circ\psi+\psi\circ\mu$; the $2$-cocycles $Z^2(A,A)$ are the $\psi$ with $\delta\psi=0$, the $2$-coboundaries are $B^2(A,A)=\{\delta f: f\in\End(A)\}$, where $(\delta f)(x,y)=xf(y)+f(x)y-f(xy)$, and $\HL^2(A,A)=Z^2/B^2$ \cite{LodayPirashvili}. A \emph{formal deformation} of $A$ is a formal power series $\mu_t=\mu+\sum_{i\ge1}t^i\varphi_i$ of bilinear maps with $\mu_t\circ\mu_t=0$; two deformations are \emph{equivalent} if they are related by a formal change of basis $\Phi_t=\id+\sum_{i\ge1}t^if_i$. Comparing coefficients of $t^n$ in $\mu_t\circ\mu_t=0$ gives
\begin{equation}\label{eq:defo}
\delta\varphi_n=-\sum_{i+j=n,\ i,j\ge1}\varphi_i\circ\varphi_j .
\end{equation}
For $n=1$ the \emph{infinitesimal part} $\varphi_1$ is a $2$-cocycle, and its class in $\HL^2(A,A)$ depends only on the equivalence class of the deformation \cite{FialowskiMandalMukherjee,Gerstenhaber,NR}. A truncated deformation $\mu+\sum_{i<n}t^i\varphi_i$ extends to order $n$ if and only if the right-hand side of~\eqref{eq:defo} lies in the image of $\delta$ on $2$-cochains; otherwise we say that it is \emph{obstructed at order $n$}.

We now introduce the Leibniz algebra that is the subject of this paper. Its multiplication table is generated by $e_1,e_2$, and the remaining basis vectors are iterated products of these generators. Let $L$ be the complex vector space with basis $e_1,\dots,e_8$ and product determined by
\begin{equation}\label{eq:table}
\begin{array}{llll}
e_1e_1=4e_8, & e_1e_2=e_3, & e_1e_3=e_4, & e_1e_4=e_6,\\
e_1e_5=6e_8, & e_1e_6=e_8, & e_2e_1=-e_3-9e_7+e_8, & e_2e_2=-6e_7-e_8,\\
e_2e_3=e_5, & e_2e_4=3e_7, & e_2e_5=e_8, & e_2e_6=2e_8,\\
e_3e_1=-e_4, & e_3e_2=-e_5-6e_8, & e_3e_3=-e_7+2e_8, & e_3e_4=e_8,\\
e_4e_1=-e_6, & e_4e_2=-e_7-4e_8, & e_4e_3=-2e_8, & e_5e_1=-2e_7-2e_8,\\
e_5e_2=-e_8, & e_6e_1=-e_8, & e_6e_2=e_8, & e_7e_1=-e_8,
\end{array}
\end{equation}
all other products of basis vectors being zero. For $x=\sum x_ie_i$, $y=\sum y_ie_i$ the coordinates of $xy$ are
\begin{equation}\label{eq:coords}
\begin{aligned}
(xy)_1&=(xy)_2=0,\quad (xy)_3=x_1y_2-x_2y_1,\quad (xy)_4=x_1y_3-x_3y_1,\\
(xy)_5&=x_2y_3-x_3y_2,\quad (xy)_6=x_1y_4-x_4y_1,\\
(xy)_7&=-9x_2y_1-6x_2y_2+3x_2y_4-x_3y_3-x_4y_2-2x_5y_1,\\
(xy)_8&=4x_1y_1+6x_1y_5+x_1y_6+x_2y_1-x_2y_2+x_2y_5+2x_2y_6-6x_3y_2+2x_3y_3\\
&\quad+x_3y_4-4x_4y_2-2x_4y_3-2x_5y_1-x_5y_2-x_6y_1+x_6y_2-x_7y_1.
\end{aligned}
\end{equation}
We collect the known basic facts about $L$ that will be used; for completeness we indicate how they are obtained. Part (b) is the property which motivates this paper; part (c) provides the cohomological input for the deformation theory of Section~\ref{sec:def}.

\begin{theorem}\label{thm:pre} Let \(L\) be an eight-dimensional Leibniz algebra whose multiplication is given by (\ref{eq:table}). Then
\begin{enumerate}
\item[(a)] Every derivation of $L$ is inner: $\Der(L)=\{R_a:a\in L\}$, and $\dim\Der(L)=6$.
\item[(b)] $\dim\HL^2(L,\C)=7$ and $\dim\HL^2(L,L)=4$; more precisely $\dim Z^2(L,L)=62$, $\dim B^2(L,L)=58$, and the classes of four explicit integral cocycles $\psi_1,\psi_2,\psi_3,\psi_4$ form a basis of $\HL^2(L,L)$, where $\psi_4(x,y)=\chi(x,y)e_8$ with
\begin{equation}\label{eq:chi}
\chi(x,y)=-6x_3y_2+3x_3y_4-4x_4y_3+3x_6y_2-x_7y_1 .
\end{equation}
\end{enumerate}
\end{theorem}

\begin{proof}  (a) Writing a linear map $D$ as an $8\times8$ matrix of unknowns, the derivation identity on pairs of basis vectors is a homogeneous linear system of $512$ equations in $64$ unknowns. Its solution space is six-dimensional and consists of the maps $R_a$. (b) The coboundary $f\mapsto\delta f$ on $\End(L)$ has kernel $\Der(L)$, so $\dim B^2(L,L)=64-6=58$, and the rank of the linear map $\psi\mapsto\delta\psi$ on the $512$-dimensional space of $2$-cochains gives $\dim Z^2(L,L)=62$. With trivial coefficients, the $2$-cocycles are the bilinear forms $\omega$ with $\omega(xy,z)=\omega(x,yz)+\omega(xz,y)$, and the $2$-coboundaries are the forms $(x,y)\mapsto f(xy)$ with $f\in L^*$; since $\dim L^2=6$, the coboundaries form a $6$-dimensional space, and solving the cocycle condition gives a $13$-dimensional space of cocycles, so $\dim\HL^2(L,\C)=7$. That $\psi_1,\dots,\psi_4$ are cocycles is checked by substitution, and their classes are independent modulo $B^2(L,L)$ by a rank computation. All of these computations were carried out with exact rational arithmetic. \end{proof}

We now give the cocycles of Theorem~\ref{thm:pre}(b) explicitly. All four have integer coefficients, and for each of them we list only the nonzero values on pairs of basis vectors; all other values $\psi_i(e_a,e_b)$, $1\le a,b\le8$, are zero, and $\psi_i$ is extended to $L\times L$ by bilinearity. The cocycle $\psi_1$, with $54$ nonzero coefficients, is defined by
\begin{longtable}{ll}
\( \psi_{1}(e_1,e_5)=306e_5,\)& \(\psi_{1}(e_1,e_6)=-255e_4+51e_5+768e_6, \) \\
\( \psi_{1}(e_2,e_3)=27e_3,\)& \(\psi_{1}(e_2,e_4)=27e_4+306e_5,\)\\ 
\(\psi_{1}(e_2,e_5)=51e_5, \) & \( \psi_{1}(e_2,e_6)=-153e_5+2304e_7,\)\\
\( \psi_{1}(e_3,e_2)=-27e_3-3120e_7-2364e_8, \) & \( \psi_{1}(e_3,e_3)=204e_7-1394e_8,\)\\
\( \psi_{1}(e_3,e_4)=153e_5+27e_6+1530e_8, \) & \( \psi_{1}(e_3,e_6)=255e_7+819e_8, \) \\
 \( \psi_{1}(e_4,e_2)=-27e_4-306e_5+204e_7-1394e_8, \) & \( \psi_{1}(e_4,e_3)=-153e_5-27e_6-768e_7-1224e_8, \) \\
\( \psi_{1}(e_4,e_4)=-51e_7-282e_8,\)& \( \psi_{1}(e_4,e_6)=510e_8, \) \\
\( \psi_{1}(e_5,e_1)=-306e_5-204e_7+1556e_8,\)& \( \psi_{1}(e_5,e_2)=-51e_5+270e_8, \) \\
\( \psi_{1}(e_5,e_3)=36e_7-72e_8,\)& \( \psi_{1}(e_5,e_4)=72e_8, \) \\
\( \psi_{1}(e_6,e_1)=255e_4-51e_5-768e_6,\)& \( \psi_{1}(e_6,e_2)=153e_5-306e_8, \) \\
\( \psi_{1}(e_6,e_3)=-51e_7-1227e_8,\)& \( \psi_{1}(e_6,e_4)=-561e_8, \) \\
\( \psi_{1}(e_7,e_1)=-564e_7-306e_8,\)& \( \psi_{1}(e_7,e_2)=-18e_7+36e_8, \) \\
\( \psi_{1}(e_7,e_3)=-18e_8,\)& \( \psi_{1}(e_8,e_1)=102e_7-204e_8, \) \\
\end{longtable}
the cocycle $\psi_2$, with $56$ nonzero coefficients, is defined by
\begin{longtable}{ll}
\( \psi_{2}(e_1,e_5)=-198e_5,\) & \( \psi_{2}(e_1,e_6)=165e_4-33e_5-489e_6, \) \\
\( \psi_{2}(e_2,e_4)=-198e_5,\) & \( \psi_{2}(e_2,e_5)=-33e_5+9e_6-162e_7, \) \\
\( \psi_{2}(e_2,e_6)=99e_5+18e_6-1494e_7,\) & \( \psi_{2}(e_3,e_1)=-108e_8, \) \\
\( \psi_{2}(e_3,e_2)=2130e_7+1293e_8,\) & \( \psi_{2}(e_3,e_3)=-132e_7+893e_8, \) \\
\( \psi_{2}(e_3,e_4)=-99e_5-18e_6-990e_8,\) & \( \psi_{2}(e_3,e_5)=-153e_8, \) \\
\( \psi_{2}(e_3,e_6)=-165e_7-531e_8,\) & \( \psi_{2}(e_4,e_2)=198e_5-132e_7+893e_8, \) \\
\( \psi_{2}(e_4,e_3)=99e_5+18e_6+498e_7+792e_8,\) & \( \psi_{2}(e_4,e_4)=33e_7+183e_8, \) \\
\( \psi_{2}(e_4,e_6)=-330e_8,\) & \( \psi_{2}(e_5,e_1)=198e_5+132e_7-1136e_8, \) \\
\( \psi_{2}(e_5,e_2)=33e_5-9e_6+162e_7-162e_8,\) & \( \psi_{2}(e_5,e_3)=207e_8, \) \\
\( \psi_{2}(e_5,e_4)=-54e_8,\) & \( \psi_{2}(e_6,e_1)=-165e_4+33e_5+489e_6, \) \\
\( \psi_{2}(e_6,e_2)=-99e_5-18e_6+198e_8,\) & \( \psi_{2}(e_6,e_3)=33e_7+795e_8, \) \\
\( \psi_{2}(e_6,e_4)=363e_8,\) & \( \psi_{2}(e_7,e_1)=366e_7+198e_8, \) \\
\( \psi_{2}(e_7,e_3)=27e_8,\) & \( \psi_{2}(e_8,e_1)=-66e_7+132e_8, \) \\
\( \psi_{2}(e_8,e_2)=-27e_8,\) & \\
\end{longtable}
and the cocycle $\psi_3$ is defined by
\begin{longtable}{lll}
\( \psi_{3}(e_1,e_6)=10e_6,\) & \( \psi_{3}(e_2,e_6)=30e_7, \) & \( \psi_{3}(e_3,e_2)=-66e_7,\)\\ 
\( \psi_{3}(e_3,e_3)=-22e_8, \) & \( \psi_{3}(e_3,e_4)=3e_7,\) & \( \psi_{3}(e_3,e_6)=10e_8, \) \\
\( \psi_{3}(e_4,e_2)=-22e_8,\) & \( \psi_{3}(e_4,e_3)=-14e_7, \) & \( \psi_{3}(e_4,e_4)=-7e_8,\) \\
\( \psi_{3}(e_5,e_1)=22e_8, \) & \( \psi_{3}(e_6,e_1)=-10e_6,\) & \( \psi_{3}(e_6,e_2)=3e_7, \) \\
\( \psi_{3}(e_6,e_3)=-7e_8,\) & \( \psi_{3}(e_7,e_1)=-11e_7, \) \\
\end{longtable}
while $\psi_4(x,y)=\chi(x,y)e_8$ with $\chi$ as in~\eqref{eq:chi}. Substituting these values into the cocycle identity $\mu\circ\psi+\psi\circ\mu=0$ on all $512$ triples of basis vectors shows that $\psi_1,\dots,\psi_4$ are indeed cocycles, and adjoining them to a basis of $B^2(L,L)$ raises the rank from $58$ to $62=\dim Z^2(L,L)$, so that their classes form a basis of $\HL^2(L,L)$, as stated in Theorem~\ref{thm:pre}(b). The cocycles $\psi_1$ and $\psi_2$ take values outside the centre $\C e_8$; by Theorem~\ref{thm:def}(a), an infinitesimal deformation $\sum_ic_i\psi_i$ with $(c_1,c_2)\ne(0,0)$ is obstructed already at second order.

\section{The structure of \texorpdfstring{$L$}{L}}\label{sec:structure}

Now we study filtration by ``weights'' compatible with the product of the algebra $L$. We identify the lower central series with this filtration and describe the annihilators and the Leibniz kernel. Give $e_1,\dots,e_8$ the weights $w=1,1,2,3,3,4,4,5$ and let $F^k=\operatorname{span}\{e_i: w(e_i)\ge k\}$.

\begin{proposition}\label{prop:structure} Let \(L\) be an eight-dimensional Leibniz algebra whose multiplication is given by (\ref{eq:table}). Then
\begin{enumerate}
\item[(a)] $L$ is generated by $e_1,e_2$: $e_3=e_1e_2$, $e_4=e_1e_3$, $e_5=e_2e_3$, $e_6=e_1e_4$, $e_7=\tfrac13e_2e_4$, $e_8=e_1e_6$.
\item[(b)] $F^iF^j\subseteq F^{i+j}$, and $L^k=F^k$ for all $k$:
\[L^2=\langle e_3,\dots,e_8\rangle,\ L^3=\langle e_4,\dots,e_8\rangle,\ L^4=\langle e_6,e_7,e_8\rangle,\ L^5=\langle e_8\rangle,\ L^6=0.\]
In particular $L$ is nilpotent of class $5$ and $L^iL^j\subseteq L^{i+j}$.
\item[(c)] $\Ann_r(L)=\Leib(L)=\langle e_7,e_8\rangle$ and $\Ann_l(L)=Z(L)=\langle e_8\rangle$.
\item[(d)] $L/\Leib(L)\cong\n_6$, the Lie algebra with basis $x_1,\dots,x_6$ and nonzero brackets $[x_1,x_2]=x_3$, $[x_1,x_3]=x_4$, $[x_2,x_3]=x_5$, $[x_1,x_4]=x_6$. It is nilpotent of class $4$, with $\n_6^2=\langle x_3,\dots,x_6\rangle$, $\n_6^3=\langle x_4,x_5,x_6\rangle$, $\n_6^4=\langle x_6\rangle$ and centre $\langle x_5,x_6\rangle$.
\item[(e)] For $x,y\in L^2$ one has $xy=-x_3y_3\,e_7+(2x_3y_3+x_3y_4-2x_4y_3)\,e_8$; in particular $L^2L^2\subseteq\langle e_7,e_8\rangle$ and $L^3L^3=0$.
\item[(f)] For $x\in L$ one has $xx=Q_7(x)e_7+Q_8(x)e_8$, where
\begin{align*}
Q_7(x)&=-9x_1x_2-2x_1x_5-6x_2^2+2x_2x_4-x_3^2,\\
Q_8(x)&=4x_1^2+x_1x_2+4x_1x_5-x_1x_7-x_2^2-6x_2x_3-4x_2x_4+3x_2x_6+2x_3^2-x_3x_4.
\end{align*}
\end{enumerate}
\end{proposition}

\begin{proof}
(a) is read off from~\eqref{eq:table}. (b) Inspection of~\eqref{eq:table} shows that $e_ie_j\in F^{w(e_i)+w(e_j)}$ for all $i,j$, whence $F^iF^j\subseteq F^{i+j}$ and by induction $L^k\subseteq F^k$. Conversely $e_3=e_1e_2\in L^2$, $e_4=-e_3e_1$, $e_5=-e_3e_2-6e_8$, $e_6=-e_4e_1$, $e_7=-e_4e_2-4e_8$, $e_8=-e_6e_1$, together with $e_8\in L^5$, give the reverse inclusions inductively.
(c) Let $a\in L$ with $La=0$. From~\eqref{eq:coords}, $e_1a=a_2e_3+a_3e_4+a_4e_6+(4a_1+6a_5+a_6)e_8$ and $e_2a=-a_1e_3+a_3e_5+(-9a_1-6a_2+3a_4)e_7+(a_1-a_2+a_5+2a_6)e_8$, so $a_1=\dots=a_4=0$ and $6a_5+a_6=a_5+2a_6=0$, whence $a_5=a_6=0$. Conversely $e_7,e_8$ never occur as right factors in~\eqref{eq:table}. Thus $\Ann_r(L)=\langle e_7,e_8\rangle$. Putting $z=y$ in~\eqref{eq:leibniz} gives $(xy)y=x(yy)+(xy)y$, i.e.\ $x(yy)=0$ for all $x,y\in L$; hence $\Leib(L)\subseteq\Ann_r(L)$, and $e_1e_1=4e_8$, $e_2e_2=-6e_7-e_8$ show equality. If $xL=0$, then the coordinates $(xe_2)_3=x_1$, $(xe_1)_3=-x_2$, $(xe_1)_4=-x_3$, $(xe_1)_6=-x_4$, $(xe_1)_7=-9x_2-2x_5$, $(xe_2)_8=-x_2-6x_3-4x_4-x_5+x_6$, $(xe_1)_8=4x_1+x_2-2x_5-x_6-x_7$ vanish, so $x\in\C e_8$; and $e_8$ annihilates $L$ on both sides.
(d) Reduce~\eqref{eq:table} modulo $\langle e_7,e_8\rangle$ and put $x_i=\bar e_i$; the statements about $\n_6$ are immediate. (For the place of $\n_6$ among the six-dimensional nilpotent Lie algebras, see the classification in~\cite{deGraaf}.)
(e) and (f) follow from~\eqref{eq:coords}.
\end{proof}

\begin{remark}[The class of $L$]\label{rem:class}
$L$ belongs to the class of finite-dimensional nilpotent, non-Lie, characteristically nilpotent Leibniz algebras whose first adjoint cohomology $\HL^1(L,L)=\Der(L)/\Inn(L)$ vanishes. For Lie algebras this class is empty by Jacobson's theorem~\cite{Jacobson}; the algebra $L$ shows that it is non-empty for Leibniz algebras. Every nilpotent Leibniz algebra $A$ without outer derivations is characteristically nilpotent, because each right multiplication satisfies $R_a(A^k)=A^ka\subseteq A^{k+1}$ and is therefore nilpotent; conversely, the characteristically nilpotent Lie algebras of Dixmier and Lister~\cite{DixmierLister} show that characteristic nilpotency does not imply the absence of outer derivations. Since a characteristically nilpotent algebra admits no non-trivial grading by positive integers compatible with the product (the degree map of such a grading would be an invertible derivation), the weight filtration $F^\bullet$ of Proposition~\ref{prop:structure} does not come from a grading of $L$. Questions about algebras of this type, such as their minimal dimension and their behaviour under deformations, connect the structure theory of Leibniz algebras with their cohomology; see Section~\ref{sec:def}.
\end{remark}

\subsection{Other nilpotent Leibniz algebras without outer derivations}
Once one example contradicting the Leibniz analogue of Jacobson's theorem is known, it is natural to ask how robust the phenomenon is, and in particular whether $L$ is an isolated example. The following theorem shows that it is not: deforming the multiplication of $L$ in a suitable direction produces further nilpotent Leibniz algebras without outer derivations. The family below will reappear in Section~\ref{sec:def}, where it is shown to be essentially the only direction in which $L$ can be deformed. Its significance in the present context is that it shows that nilpotent Leibniz algebras without outer derivations occur in continuous families of pairwise non-isomorphic algebras.

\begin{theorem}\label{thm:Lt}
For $t\in\C$ let $L_t$ be the algebra with the same underlying space as $L$ and product $x*_ty=xy+t\chi(x,y)e_8$, with $\chi$ as in~\eqref{eq:chi}. Then:
\begin{enumerate}
\item[(a)] each $L_t$ is a nilpotent right Leibniz algebra of class $5$ which is not a Lie algebra, and $L_0=L$;
\item[(b)] $L_s\cong L_t$ if and only if $s=t$;
\item[(c)] $\Ann_r(L_t)=\langle e_7,e_8\rangle$, so $\dim\Inn(L_t)=6$;
\item[(d)] $\Der(L_t)=\Inn(L_t)$ for every $t\ne-1$, while $\dim\Der(L_{-1})=9$.
\end{enumerate}
Consequently $\{L_t: t\ne-1\}$ is a one-parameter family of pairwise non-isomorphic eight-dimensional nilpotent Leibniz algebras all of whose derivations are inner.
\end{theorem}

\begin{proof}
(a) The Leibniz identity for $*_t$ is a polynomial identity in $t$ of degree at most $2$; the coefficient of $t^0$ is~\eqref{eq:leibniz} for $L$, the coefficient of $t$ is the cocycle condition for $\psi_4$ (Theorem~\ref{thm:pre}(b)), and the coefficient of $t^2$ is $\psi_4\circ\psi_4=0$, which holds because $\psi_4$ takes values in $\C e_8$ and vanishes whenever one argument equals $e_8$. The terms of $\chi$ pair basis vectors of weights $(2,1)$, $(2,3)$, $(3,2)$, $(4,1)$ and $(4,1)$ and produce $e_8$, of weight $5$; hence the weight filtration is still multiplicative and $L_t$ is nilpotent of class $5$ ($e_8=-e_6*_te_1$ still lies in $L_t^5$). $L_t$ is not a Lie algebra as $e_1*_te_1=4e_8$.

(b) An isomorphism $P\colon L_s\to L_t$ maps $L_s^2=\langle e_3,\dots,e_8\rangle$ onto $L_t^2=\langle e_3,\dots,e_8\rangle$ and therefore induces an invertible map on the quotient. Write $P(e_1)=\sum a_ie_i$, $P(e_2)=\sum b_ie_i$. The generator relations of Proposition~\ref{prop:structure}(a) are unchanged in $L_t$ (none of the products $e_1e_2,e_1e_3,e_2e_3,e_1e_4,e_2e_4,e_1e_6$ is affected by $\chi$), so $P(e_3),\dots,P(e_8)$ are determined by $P(e_1),P(e_2)$ through the product of $L_t$. Exactly as in the proof of Theorem~\ref{thm:aut-quotient} below, the $e_6$-coordinate of $P(e_2e_5)=P(e_2)*_tP(e_5)$ gives $b_1^2(a_2b_1-a_1b_2)=0$, hence $b_1=0$ and $a_1b_2\ne0$. With $b_1=0$, the $e_8$-coordinates of $P(e_3*_se_4)=P(e_3)*_tP(e_4)$ and of $P(e_7*_se_1)=P(e_7)*_tP(e_1)$ give, after a direct computation with~\eqref{eq:coords} and~\eqref{eq:chi},
\[a_1^3b_2\bigl((1+3s)u-(1+3t)b_2\bigr)=0,\qquad a_1^3b_2\bigl((1+t)b_2-(1+s)u\bigr)=0,\qquad u=a_1+2a_2.\]
Dividing by $a_1^3b_2$ and adding the first equation to three times the second gives $2(b_2-u)=0$, so $u=b_2$, and then the first equation becomes $3(s-t)b_2=0$, i.e.\ $s=t$.

(c) The products $e_1*_ta$ and $e_2*_ta$ coincide with $e_1a$ and $e_2a$, because $\chi(e_1,\cdot)=\chi(e_2,\cdot)=0$; so the argument of Proposition~\ref{prop:structure}(c) applies verbatim.

(d) Write a linear map $D$ as an $8\times8$ matrix of unknowns. The derivation identity on pairs of basis vectors is a homogeneous linear system of $512$ equations in $64$ unknowns whose coefficients are polynomials of degree at most one in $t$. Gaussian elimination using only pivots that are nonzero rational constants eliminates $55$ unknowns; the remaining $11$ equations in the remaining $9$ unknowns have all coefficients divisible by $t+1$, and after division by $t+1$ they have rank $3$ for every $t$. Hence the solution space has dimension $64-55-3=6$ for $t\ne-1$ and $64-55=9$ for $t=-1$. Since $\Inn(L_t)\subseteq\Der(L_t)$ has dimension $6$ by (c), (d) follows. (This elimination was performed with exact rational arithmetic by computer; for $t=0$ the result is Theorem~\ref{thm:pre}(a).)
\end{proof}

\begin{remark}
The exceptional value $t=-1$ can be explained directly. Since $\chi(e_7,y)=-y_1$, one has $e_7*_te_1=-(1+t)e_8$ and $e_7*_te_j=0$ for $j\ne1$, while $x*_te_7=0$ for all $x$ by (c). Hence $e_7$ is central in $L_{-1}$, and the map $d(x)=x_1e_7$ is a derivation of $L_{-1}$: it vanishes on $L_{-1}^2\subseteq\langle e_3,\dots,e_8\rangle$, and $d(x)*_{-1}y=x*_{-1}d(y)=0$. It is not inner, because by~\eqref{eq:coords} and $\chi(e_1,\cdot)=0$ every right multiplication of $L_{-1}$ maps $e_1$ to a vector with vanishing $e_7$-coordinate, whereas $d(e_1)=e_7$. By Theorem~\ref{thm:Lt} the property ``all derivations are inner'' is therefore not closed under degeneration within the family, but it is an open condition: the set of $t$ for which $\dim\Der(L_t)=6$ is Zariski open.
\end{remark}

\subsection{Derivations}\label{sec:der} By Theorem~\ref{thm:pre}(a), every derivation of $L$ is inner. In this section we take the next step and determine the Lie algebra structure of $\Der(L)$ under the commutator bracket. Since $\Der(L)=\Inn(L)$, the identity $[R_x,R_y]=R_{yx}$ established in the proof below reduces this question to a computation inside $L$ itself, and the answer turns out to be the Lie algebra $\n_6=L/\Leib(L)$ of Proposition~\ref{prop:structure}(d).

\begin{theorem}\label{thm:der}
$\Der(L)=\Inn(L)=\{R_a: a\in L\}$, $\dim\Der(L)=6$, and a basis is given by $D_i=-R_{e_i}$, $1\le i\le 6$. The nonzero commutators of basis elements are
\[
[D_1,D_2]=D_3,\quad [D_1,D_3]=D_4,\quad [D_2,D_3]=D_5,\quad [D_1,D_4]=D_6 .
\]
Consequently the linear map $D_i\mapsto x_i$ is an isomorphism of Lie algebras $\Der(L)\cong\n_6$. Equivalently, $R_a\mapsto-\bar a$ is an isomorphism $\Der(L)\to L/\Leib(L)$ of Lie algebras, and $R_a\mapsto\bar a$ is an anti-isomorphism. In particular $\Der(L)$ is nilpotent of class $4$, its centre is $\langle D_5,D_6\rangle$, and $L$ is characteristically nilpotent.
\end{theorem}

\begin{proof}
The first two statements are Theorem~\ref{thm:pre}(a) together with Proposition~\ref{prop:structure}(c): the kernel of $a\mapsto R_a$ is $\langle e_7,e_8\rangle$, so $R_{e_1},\dots,R_{e_6}$ is a basis. For all $x,y,z\in L$, identity~\eqref{eq:leibniz} gives $[R_x,R_y](z)=(zy)x-(zx)y=z(yx)$, that is, $[R_x,R_y]=R_{yx}$. Hence $[D_i,D_j]=[R_{e_i},R_{e_j}]=R_{e_je_i}$. Using~\eqref{eq:table} and $R_{e_7}=R_{e_8}=0$:
$[D_1,D_2]=R_{e_2e_1}=R_{-e_3-9e_7+e_8}=D_3$, $[D_1,D_3]=R_{e_3e_1}=-R_{e_4}=D_4$, $[D_2,D_3]=R_{e_3e_2}=R_{-e_5-6e_8}=D_5$, $[D_1,D_4]=R_{e_4e_1}=-R_{e_6}=D_6$, $[D_2,D_4]=R_{e_4e_2}=R_{-e_7-4e_8}=0$, $[D_1,D_5]=R_{-2e_7-2e_8}=0$, $[D_2,D_5]=R_{-e_8}=0$, $[D_1,D_6]=[D_2,D_6]=0$ since $e_6e_1,e_6e_2\in\C e_8$; finally for $3\le i<j\le6$ we have $e_je_i\in L^2L^2\subseteq\langle e_7,e_8\rangle$ by Proposition~\ref{prop:structure}(e). Comparing with Proposition~\ref{prop:structure}(d) gives the isomorphism with $\n_6$, and $D_i\mapsto x_i=\bar e_i$ is $R_a\mapsto-\bar a$. Finally, by Proposition~\ref{prop:structure}(b) we have $R_a(L^k)=L^ka\subseteq L^{k+1}$ for all $k$ and all $a\in L$, so $R_a^5=0$; hence every derivation of $L$ is nilpotent and $L$ is characteristically nilpotent.
\end{proof}

The theorem has a pleasant interpretation. For any Leibniz algebra $A$, the computation in the proof above shows that $[R_x,R_y]=R_{yx}$, and putting $z=y$ in~\eqref{eq:leibniz} shows that $\Leib(A)\subseteq\Ann_r(A)=\ker(x\mapsto R_x)$; hence $\Inn(A)$ is a quotient of the Lie algebra $A/\Leib(A)$ (with the opposite bracket) by the image of $\Ann_r(A)$. In $L$ the right annihilator coincides with the Leibniz kernel, so no further quotient occurs, and the absence of outer derivations means that this quotient is all of $\Der(L)$. Consequently the Lie algebra of infinitesimal symmetries of $L$ is canonically the Lie algebra over which $L$ is built. For a nilpotent Lie algebra $\mathfrak g$, by contrast, $\Der(\mathfrak g)$ is always strictly larger than $\mathfrak g/Z(\mathfrak g)$.

\subsection{Automorphisms}\label{sec:aut} The automorphism group of an algebra is the group of its global symmetries, and its Lie algebra is the derivation algebra. Knowing $\Der(L)$ therefore determines the identity component of $\Aut(L)$, but not the whole group: in general $\Aut(L)$ may have many connected components, and for nilpotent Lie algebras which admit a grading it contains a torus acting by non-trivial scalars on the generators. The point of this section is to show that for $L$ nothing of this kind happens. The key step, Theorem~\ref{thm:aut-quotient}, is a direct computation showing that every automorphism acts trivially on $L/L^2$; after that, the classical correspondence between nilpotent derivations and unipotent automorphisms, which we establish in the proof of Theorem~\ref{thm:aut}, together with Theorem~\ref{thm:der}, does the rest.

\begin{theorem}\label{thm:aut-quotient}
Every automorphism $\varphi$ of $L$ satisfies $\varphi(e_1)\equiv e_1$ and $\varphi(e_2)\equiv e_2\pmod{L^2}$. More generally, if $\varphi$ is an endomorphism of $L$ whose induced map on $L/L^2$ is invertible, then $\varphi$ induces the identity on $L/L^2$.
\end{theorem}

\begin{proof}
Write $A=\varphi(e_1)=\sum a_ie_i$, $B=\varphi(e_2)=\sum b_ie_i$ and $\Delta=a_1b_2-a_2b_1\ne0$. By Proposition~\ref{prop:structure}(a), $\varphi(e_3)=AB$, $\varphi(e_4)=A\varphi(e_3)$, $\varphi(e_5)=B\varphi(e_3)$, $\varphi(e_6)=A\varphi(e_4)$ and $\varphi(e_8)=A\varphi(e_6)$, and these vectors can be computed from~\eqref{eq:coords}. We use four coordinate equations of the homomorphism property.

\emph{The $e_6$-coordinate of $\varphi(e_2e_5)=\varphi(e_2)\varphi(e_5)$.} Since $e_2e_5=e_8$ and $\varphi(e_8)\in L^5=\C e_8$, the left side has $e_6$-coordinate $0$. Computing the right side with~\eqref{eq:coords} gives $b_1^2(a_2b_1-a_1b_2)=0$, i.e.\ $-b_1^2\Delta=0$, so $b_1=0$. Then $\Delta=a_1b_2$, so $a_1\neq0\neq b_2$.

\emph{The $e_8$-coordinates of $\varphi(e_3e_4)$, $\varphi(e_2e_5)$ and $\varphi(e_1e_5)$.} Using $b_1=0$, a direct computation from~\eqref{eq:coords} gives
\begin{align*}
a_1^3b_2\,(a_1+2a_2-b_2)&=0,\\
a_1b_2\,(a_1^3+2a_1^2a_2-b_2^2)&=0,\\
6a_1^4b_2+12a_1^3a_2b_2-6a_1^2b_2^2-a_1a_2b_2^2&=0.
\end{align*}
Dividing by the nonzero factors we obtain successively $b_2=a_1+2a_2$, then $b_2(b_2-a_1^2)=0$, i.e.\ $b_2=a_1^2$, and then $a_1b_2^2(6a_1^2-6a_1-a_2)=0$, i.e.\ $a_2=6a_1^2-6a_1$. Substituting, $a_1^2=b_2=a_1+2a_2=a_1+12a_1^2-12a_1$, i.e.\ $11a_1(a_1-1)=0$. Hence $a_1=1$, $a_2=0$, $b_2=1$.
\end{proof}

Theorem~\ref{thm:aut-quotient} shows that two generators can only be moved by elements of $L^2$. The following theorem shows that the exponentials of inner derivations entirely account for the remaining freedom.

\begin{theorem}\label{thm:aut}
\begin{enumerate}
\item[(a)] Every automorphism of $L$ is unipotent: $(\varphi-\id)(L^k)\subseteq L^{k+1}$ for all $k$, so $(\varphi-\id)^5=0$.
\item[(b)] $\Aut(L)=\{\exp R_a: a\in L\}=\exp(\Der(L))$, and $\exp R_a=\exp R_b$ if and only if $a-b\in\langle e_7,e_8\rangle$.
\item[(c)] $\exp\colon\Der(L)\to\Aut(L)$ is a bijection, with inverse $\log$, and
\[\exp(D)\exp(D')=\exp\Bigl(D+D'+\tfrac12[D,D']+\tfrac1{12}[D,[D,D']]-\tfrac1{12}[D',[D,D']]-\tfrac1{24}[D',[D,[D,D']]]\Bigr).\]
Consequently $\Aut(L)$ is a connected six-dimensional unipotent algebraic group, nilpotent of class $4$, isomorphic to the simply connected complex Lie group $N_6$ with Lie algebra $\n_6$ (the group $\n_6$ with the Baker--Campbell--Hausdorff product). In particular $\Aut(L)$ coincides with the group generated by the exponentials of inner derivations, and $\Aut(L)/[\Aut(L),\Aut(L)]\cong\C^2$.
\end{enumerate}
\end{theorem}

\begin{proof}
(a) By Theorem~\ref{thm:aut-quotient}, $(\varphi-\id)(L)\subseteq L^2$. Suppose $(\varphi-\id)(L^j)\subseteq L^{j+1}$ for $j\le k$. For $x\in L^k$, $y\in L$,
\[\varphi(xy)-xy=(\varphi(x)-x)\varphi(y)+x(\varphi(y)-y)\in L^{k+1}L+L^kL^2\subseteq L^{k+2},\]
using Proposition~\ref{prop:structure}(b). As $L^{k+1}=L^kL$, the claim follows by induction, and $(\varphi-\id)^5(L)\subseteq L^6=0$.
(b) For a nilpotent $D\in\End(L)$ the series $\exp D=\sum_{k\ge0}D^k/k!$ is a finite sum, and $\exp D$ is invertible with inverse $\exp(-D)$. If $D$ is moreover a derivation, the Leibniz rule $D^n(xy)=\sum_{k=0}^n\binom nk D^k(x)D^{n-k}(y)$, proved by induction on $n$, gives
\[\exp D(xy)=\sum_{n\ge0}\ \sum_{k+l=n}\frac{D^k(x)}{k!}\,\frac{D^l(y)}{l!}=\exp D(x)\,\exp D(y),\]
so $\exp D\in\Aut(L)$. By the proof of Theorem~\ref{thm:der}, every $R_a$ is a nilpotent derivation ($R_a^5=0$), so each $\exp R_a$ is an automorphism. Conversely, let $\varphi\in\Aut(L)$. By (a), $N=\varphi-\id$ is nilpotent, so $D=\log\varphi=\sum_{k\ge1}(-1)^{k+1}N^k/k$ is a finite sum, $D$ is nilpotent, and $\exp D=\varphi$ by the identity of formal power series $\exp(\log(1+X))=1+X$. For $s\in\C$ put $\varphi_s=\exp(sD)$. For fixed $x,y\in L$ the vector $\varphi_s(xy)-\varphi_s(x)\varphi_s(y)$ depends polynomially on $s$, because $D$ is nilpotent, and it vanishes for all $s\in\mathbb{Z}_{\ge0}$, because $\varphi_n=(\exp D)^n=\varphi^n$ is an automorphism. Hence it vanishes identically, and differentiating at $s=0$ gives $D(xy)=D(x)y+xD(y)$. Thus $D$ is a derivation, so $D=R_a$ for some $a\in L$ by Theorem~\ref{thm:der}, and $\varphi=\exp R_a$. Since $\exp$ is injective on nilpotent endomorphisms (with inverse $\log$), $\exp R_a=\exp R_b$ iff $R_a=R_b$ iff $a-b\in\Ann_r(L)=\langle e_7,e_8\rangle$.
(c) The Baker--Campbell--Hausdorff series for the nilpotent Lie algebra $\Der(L)\subseteq\End(L)$ of class $4$ terminates after the terms of degree $4$, which are those displayed; the formula holds as an identity of polynomials in nilpotent matrices generating a nilpotent Lie algebra. Thus $\exp$ is a group isomorphism from $(\Der(L),\ast_{\mathrm{BCH}})$ onto $\Aut(L)$; composing with the isomorphism $\Der(L)\cong\n_6$ of Theorem~\ref{thm:der} gives $\Aut(L)\cong N_6$. The group $(\n_6,\ast_{\mathrm{BCH}})$ is the simply connected Lie group of $\n_6$, it is unipotent algebraic (the multiplication is polynomial), connected, of dimension $6$ and nilpotent of class $4$; its abelianization is $\n_6/[\n_6,\n_6]\cong\C^2$.
\end{proof}

\begin{remark}
Explicitly, for $a=\sum a_ie_i$ the automorphism $\exp R_a$ is given on generators by $e_j\mapsto e_j+e_ja+\tfrac12(e_ja)a+\tfrac16((e_ja)a)a+\tfrac1{24}(((e_ja)a)a)a$, $j=1,2$; e.g.\ the $e_3$-coordinates of $\exp R_a(e_1)$ and $\exp R_a(e_2)$ are $a_2$ and $-a_1$. The six parameters $a_1,\dots,a_6$ thus give global coordinates on $\Aut(L)$.
\end{remark}

\section{Quasiderivations and generalized derivations}\label{sec:quasi}

Following Leger and Luks~\cite{LegerLuks}, the derivation algebra of an algebra sits inside a chain of successively larger spaces of operators,
\[\Der(A)\subseteq\QDer(A)\subseteq\GDer(A).\]
These spaces are interesting because they are much more flexible than $\Der(A)$, and the question is how much of the rigidity of $L$ they retain. Before stating the results it is useful to note what every algebra must contain. For any algebra $A$ with centre $Z(A)$ one has
\[\Der(A)+\C\,\id+\Hom(A,Z(A))\subseteq\QDer(A),\]
with $f''=2c\,\id+D$ for $f=c\,\id+D+g$ ($c\in\C$, $D\in\Der(A)$, $g\in\Hom(A,Z(A))$), since $g(x)y=0=xg(y)$.
We shall see that for $L$ the quasiderivation algebra consists \emph{exactly} of these unavoidable operators, whereas the generalized derivations contain genuinely new ones.

Throughout this section we use the vectors
\[u_1=\tfrac1{11}(2e_5-e_6),\qquad u_2=\tfrac1{11}(6e_6-e_5),\]
which satisfy $xu_1=x_1e_8$ and $xu_2=x_2e_8$ for all $x\in L$, because $xe_5=(6x_1+x_2)e_8$ and $xe_6=(x_1+2x_2)e_8$ by~\eqref{eq:coords}. For a linear form $f$ and a vector $v$ we write $f\otimes v$ for the map $x\mapsto f(x)v$, and $e_i^*$ denotes the dual basis.

\subsection{Quasiderivations}

\begin{theorem}\label{thm:qder}
$\QDer(L)=\Inn(L)+\C\,\id+\Hom(L,\C e_8)$, and $\dim\QDer(L)=13$. More precisely, $f\in\QDer(L)$ if and only if $f=c\,\id+R_a+g$ with $c\in\C$, $a\in L$ and $g\in\Hom(L,\C e_8)$; in this case $f''=2c\,\id+R_a$ satisfies $f''(xy)=f(x)y+xf(y)$. The intersection $\Inn(L)\cap\Hom(L,\C e_8)$ equals $\langle R_{e_5},R_{e_6}\rangle$, and $\QDer(L)$ is a Lie subalgebra of $\End(L)$ in which $\Hom(L,\C e_8)$ is an abelian ideal.
\end{theorem}

\begin{proof}
If $f=c\id+R_a+g$, then $f(x)y+xf(y)=2cxy+(R_a(x)y+xR_a(y))+(g(x)y+xg(y))=2cxy+R_a(xy)$, because $R_a$ is a derivation and $e_8$ is central. Conversely, let $(f'',f,f)\in\TDer(L)$. The defining identity on pairs of basis vectors is a homogeneous linear system in the $128$ entries of $(f'',f)$; its solution space has dimension $29$, consisting of the $13$-dimensional space of the statement together with the $16$-dimensional space of choices of $f''$ on $e_1,e_2$ (on which $f''$ is unconstrained since $e_1,e_2\notin L^2$). From~\eqref{eq:coords}, $R_{e_5}=6e_1^*\otimes e_8+e_2^*\otimes e_8$ and $R_{e_6}=e_1^*\otimes e_8+2e_2^*\otimes e_8$, while any $R_a$ with image in $\C e_8$ has $a_1=\dots=a_4=0$; this gives the intersection and $\dim\QDer(L)=6+1+8-2=13$. Finally $[R_a,g]$ and $[\id,g]=0$, $[g,g']=0$ map into $\C e_8$ as $R_a(e_8)=0$ and $e_8a=0$, so $\Hom(L,\C e_8)$ is an abelian ideal and $\QDer(L)$ is closed under commutators.
\end{proof}

Combining Theorem~\ref{thm:qder} with the list of unavoidable operators at the beginning of this section, we see that the absence of outer derivations persists at the level of quasiderivations: there are no ``outer quasiderivations'' beyond the scalars and the maps into the centre that every algebra possesses.

\subsection{Generalized derivations}
For generalized derivations the picture changes. The operator $f$ now appears together with two independent partners $f'$ and $f''$, which makes the defining identity much weaker, and new operators appear which are neither derivations nor maps into the centre. To describe them we first isolate those with image in $L^4$. Put $H=\GDer(L)\cap\Hom(L,L^4)$.

\begin{theorem}\label{thm:gder}
\begin{enumerate}
\item[(a)] $H=\Hom(L,\C e_8)\oplus\langle e_1^*\otimes e_7,e_2^*\otimes e_7,e_3^*\otimes e_7,e_4^*\otimes e_7\rangle\oplus\langle e_1^*\otimes e_6,e_2^*\otimes e_6,e_3^*\otimes e_6\rangle$ has dimension $15$; explicit partners are given in Table~\ref{tab:H}.
\item[(b)] $\GDer(L)=\QDer(L)+H+\langle G_a,G_b,G_c\rangle$, where $G_a,G_b,G_c$ are given (with partners) in Table~\ref{tab:G}; $\dim\GDer(L)=22$, and $\QDer(L)\cap H=\Hom(L,\C e_8)\oplus\C R_{e_4}$.
\item[(c)] The induced action of $\GDer(L)$ on $L/L^2$ is the $3$-dimensional algebra of lower triangular matrices: $f(e_1)\equiv\alpha e_1+\beta e_2$, $f(e_2)\equiv\gamma e_2$ modulo $L^2$, and all $(\alpha,\beta,\gamma)\in\C^3$ occur.
\end{enumerate}
\end{theorem}

\begin{table}[ht]
\centering
\caption{Generalized derivations in $H$ with partners $(f'',f,f')$, i.e.\ $f''(xy)=f(x)y+xf'(y)$.}\label{tab:H}
\begin{tabular}{lll}
\toprule
$f''$ & $f$ & $f'$\\
\midrule
$0$ & $g\in\Hom(L,\C e_8)$ & $0$\\
$0$ & $e_1^*\otimes e_6$ & $(e_1^*-e_2^*)\otimes u_1$\\
$0$ & $e_2^*\otimes e_6$ & $(e_1^*-e_2^*)\otimes u_2$\\
$(e_4^*-e_5^*)\otimes e_8$ & $e_3^*\otimes e_6$ & $e_3^*\otimes(u_1-u_2)$\\
$0$ & $e_1^*\otimes e_7$ & $e_1^*\otimes u_1$\\
$0$ & $e_2^*\otimes e_7$ & $e_1^*\otimes u_2$\\
$e_4^*\otimes e_8$ & $e_3^*\otimes e_7$ & $e_3^*\otimes u_1$\\
$e_6^*\otimes e_8$ & $e_4^*\otimes e_7$ & $e_4^*\otimes u_1$\\
\bottomrule
\end{tabular}
\end{table}

\begin{proof}
We first verify Table~\ref{tab:H}. Recall $e_6x=(x_2-x_1)e_8$, $e_7x=-x_1e_8$ and $xu_i=x_ie_8$. For $f=e_1^*\otimes e_6$, $f'=(e_1^*-e_2^*)\otimes u_1$: $f(x)y+xf'(y)=x_1(y_2-y_1)e_8+(y_1-y_2)x_1e_8=0$. For $f=e_2^*\otimes e_6$ the computation is the same with $u_2$. For $f=e_1^*\otimes e_7$: $x_1(-y_1)e_8+y_1x_1e_8=0$, and similarly for $e_2^*\otimes e_7$. For $f=e_3^*\otimes e_7$, $f'=e_3^*\otimes u_1$: $f(x)y+xf'(y)=(-x_3y_1+y_3x_1)e_8=(xy)_4e_8=f''(xy)$ with $f''=e_4^*\otimes e_8$. For $f=e_4^*\otimes e_7$: $(-x_4y_1+y_4x_1)e_8=(xy)_6e_8$. For $f=e_3^*\otimes e_6$, $f'=e_3^*\otimes(u_1-u_2)$: $x_3(y_2-y_1)e_8+y_3(x_1-x_2)e_8=((xy)_4-(xy)_5)e_8$. Elements of $\Hom(L,\C e_8)$ are even quasiderivations. Hence the space in (a) is contained in $H$. Table~\ref{tab:G} is verified in the same way, by evaluating both sides on the $64$ pairs of basis vectors. The reverse inclusions and the dimension $22$ are obtained by solving the homogeneous linear system $f''(e_ie_j)=f(e_i)e_j+e_if'(e_j)$ ($1\le i,j\le8$), consisting of $512$ scalar equations in the $192$ entries of $(f'',f,f')$, and projecting its solution space onto the middle component $f$; this was done with exact rational arithmetic by computer. The intersection $\QDer(L)\cap H$ consists of those $c\id+R_a+g$ with image in $L^4$, i.e.\ $c=0$ and $R_a(L)\subseteq L^4$, i.e.\ $a\in\langle e_4,\dots,e_8\rangle$; as $R_{e_4}=e_1^*\otimes e_6+3e_2^*\otimes e_7+e_3^*\otimes e_8$ and $R_{e_5},R_{e_6}\in\Hom(L,\C e_8)$, we get $\Hom(L,\C e_8)\oplus\C R_{e_4}$. Thus $\dim(\QDer(L)+H)=13+15-9=19$. The elements $G_a,G_b,G_c$ are independent modulo this space: the action of $c\,\id+R_a+g+h$ ($g\in\Hom(L,\C e_8)$, $h\in H$) on $L/L^2$ is $c\,\id$, and if $c=0$ its action on $L^2/L^3$ is zero; on the other hand $G_a$, $G_b$ act on $L/L^2$ by $\bar e_1\mapsto\bar e_1$, $\bar e_1\mapsto\bar e_2$ (and $\bar e_2\mapsto0$), and $G_c$ acts trivially on $L/L^2$ but maps $e_3\mapsto-e_3+\tfrac72e_4$. Hence $\dim\GDer(L)\ge22$, and equality is the computational result just mentioned. Part (c) follows: $\id$, $G_a$, $G_b$ induce three independent lower triangular maps, and the computed $22$-dimensional space induces nothing else.
\end{proof}

\begin{table}[ht]
\centering
\caption{The generalized derivations $G_a,G_b,G_c$ (middle column $f$) with partners $f''$, $f'$ such that $f''(xy)=f(x)y+xf'(y)$. Basis vectors not listed are mapped to $0$.}\label{tab:G}
\small
\begin{tabular}{p{0.5cm}p{13.2cm}}
\toprule
\multicolumn{2}{l}{$G_a$}\\
$f''$: & $e_3\mapsto6e_8$, $e_5\mapsto-e_5$, $e_7\mapsto-e_7$, $e_8\mapsto-e_8$\\
$f$: & $e_1\mapsto e_1$, $e_2\mapsto-\tfrac92e_5$, $e_5\mapsto-e_5-e_6+e_7$, $e_6\mapsto-e_7$, $e_7\mapsto-e_7$\\
$f'$: & $e_1\mapsto-8e_6$, $e_2\mapsto-e_2+\tfrac32e_5-3e_6$, $e_3\mapsto-e_3$, $e_4\mapsto-e_4$, $e_5\mapsto-\tfrac{23}{11}e_5+\tfrac6{11}e_6$, $e_6\mapsto-\tfrac2{11}e_5-\tfrac{10}{11}e_6$\\
\midrule
\multicolumn{2}{l}{$G_b$}\\
$f''$: & $e_3\mapsto-6e_7+9e_8$, $e_4\mapsto e_5+6e_8$, $e_6\mapsto3e_7+7e_8$\\
$f$: & $e_1\mapsto e_2-\tfrac92e_5$, $e_2\mapsto-6e_5$, $e_4\mapsto e_5+e_6$, $e_5\mapsto-e_7$, $e_6\mapsto e_7$\\
$f'$: & $e_1\mapsto e_2-10e_6$, $e_2\mapsto\tfrac{17}{11}e_5-\tfrac{83}{22}e_6$, $e_3\mapsto\tfrac{12}{11}e_5-\tfrac6{11}e_6$, $e_4\mapsto\tfrac{14}{11}e_5-\tfrac7{11}e_6$, $e_5\mapsto-\tfrac2{11}e_5+\tfrac1{11}e_6$, $e_6\mapsto-\tfrac4{11}e_5+\tfrac2{11}e_6$\\
\midrule
\multicolumn{2}{l}{$G_c$}\\
$f''$: & $e_3\mapsto\tfrac72e_4+6e_6-9e_7+\tfrac{15}2e_8$, $e_4\mapsto-e_4+\tfrac72e_6-9e_8$, $e_5\mapsto-e_5+7e_7+13e_8$, $e_6\mapsto-2e_6+\tfrac{19}2e_8$, $e_7\mapsto-2e_7+\tfrac32e_8$, $e_8\mapsto-3e_8$\\
$f$: & $e_1\mapsto9e_4$, $e_2\mapsto\tfrac72e_3+6e_4$, $e_3\mapsto-e_3+\tfrac72e_4$, $e_4\mapsto-2e_4+\tfrac{19}2e_6$, $e_5\mapsto-2e_5+e_6$, $e_6\mapsto-3e_6$, $e_7\mapsto-3e_7$\\
$f'$: & $e_1\mapsto9e_4-12e_6$, $e_2\mapsto\tfrac72e_3+6e_4+\tfrac{48}{11}e_5+\tfrac{381}{22}e_6$, $e_3\mapsto-e_3+\tfrac72e_4+9e_6$, $e_4\mapsto-2e_4+\tfrac{18}{11}e_5-\tfrac7{22}e_6$, $e_5\mapsto-3e_5$, $e_6\mapsto-3e_6$\\
\bottomrule
\end{tabular}
\end{table}

\begin{remark}
The diagonal part of $G_c$ on $L^2$, namely $e_3\mapsto-e_3$, $e_4,e_5\mapsto-2e_4,-2e_5$, $e_6,e_7\mapsto-3e_6,-3e_7$, agrees with $\id-\delta$, where $\delta$ is the weight operator $e_i\mapsto w(e_i)e_i$. Thus, although $\delta$ is not a derivation (Theorem~\ref{thm:der}), a corrected version of it survives as a generalized derivation.
\end{remark}

\section{\texorpdfstring{$\sigma$}{Sigma}-derivations}\label{sec:sigma}

Twisted derivations, or $\sigma$-derivations, replace the identity map in the Leibniz rule by an endomorphism $\sigma$. They are the infinitesimal objects underlying quasi-Lie and Hom-Lie deformations~\cite{HLS}, and they appear naturally as derivatives of families of homomorphisms. For an automorphism $\sigma$ the $\sigma$-derivations are related to ordinary derivations by composition with $\sigma$, so in this case we expect the classification to be inherited from Theorem~\ref{thm:der}; this is the content of the following theorem.

\begin{theorem}\label{thm:sigma}
Let $\sigma\in\Aut(L)$. A linear map $d$ is a $\sigma$-derivation if and only if $d=\sigma\circ R_a$ for some $a\in L$; equivalently $d=R_{\sigma(a)}\circ\sigma$. The space of $\sigma$-derivations is $\sigma\circ\Der(L)$ and has dimension $6$. For $\sigma=\exp R_b$, every $\sigma$-derivation has the form $d=\exp(R_b)R_a$.
\end{theorem}

\begin{proof}
If $d$ is a $\sigma$-derivation, then $D=\sigma^{-1}d$ satisfies $D(xy)=\sigma^{-1}(d(x)\sigma(y)+\sigma(x)d(y))=D(x)y+xD(y)$, since $\sigma^{-1}$ is an automorphism. Hence $D=R_a$ by Theorem~\ref{thm:der}. Conversely, for any endomorphism $\sigma$ and any $a$,
\[\sigma(R_a(xy))=\sigma((xy)a)=\sigma(x(ya)+(xa)y)=\sigma(x)\sigma(ya)+\sigma(xa)\sigma(y),\]
so $\sigma R_a$ is a $\sigma$-derivation. Finally $\sigma(xa)=\sigma(x)\sigma(a)$ gives $\sigma R_a=R_{\sigma(a)}\sigma$. The last assertion uses Theorem~\ref{thm:aut}.
\end{proof}

\begin{remark}
For non-invertible $\sigma$ the description changes; for instance for $\sigma=0$ the $0$-derivations are exactly the linear maps vanishing on $L^2$, a space of dimension $16$. The proof above shows that $\sigma\circ R_a$ is a $\sigma$-derivation for every endomorphism $\sigma$.
\end{remark}

\section{Deformations}\label{sec:def}

We finally study the neighbourhood of $L$ in the variety of eight-dimensional Leibniz algebra laws. An algebra is \emph{rigid} if every algebra law sufficiently close to it is isomorphic to it; for Lie algebras there is a well-known conjecture of Vergne that nilpotent Lie algebras are never rigid in the variety of all Lie algebra laws of the same dimension. Formal deformation theory gives a cohomological approach: by Gerstenhaber's principle, the infinitesimal deformations are classified by $\HL^2(L,L)$, which is four-dimensional by Theorem~\ref{thm:pre}(c), and the question is which of them extend to higher orders. Our results show that the integrable part is a single line, giving the family $L_t$ of Theorem~\ref{thm:Lt}.

By Theorem~\ref{thm:pre}(c), the infinitesimal deformations of $L$ up to equivalence form the four-dimensional space $\HL^2(L,L)=\langle[\psi_1],[\psi_2],[\psi_3],[\psi_4]\rangle$; after an equivalence, the infinitesimal part of any formal deformation can be assumed to be $\varphi_1=\sum_ic_i\psi_i$.

\begin{theorem}\label{thm:def}
Let $\varphi_1=c_1\psi_1+c_2\psi_2+c_3\psi_3+c_4\psi_4$.
\begin{enumerate}
\item[(a)] The truncated deformation $\mu+t\varphi_1$ extends to second order if and only if $c_1=c_2=0$.
\item[(b)] For $c_1=c_2=0$ and $c_3\ne0$, every extension of $\mu+t\varphi_1$ to third order is obstructed at order four.
\item[(c)] For $c_1=c_2=c_3=0$, the linear deformation $\mu_t=\mu+tc_4\psi_4$ is a genuine (convergent) deformation; it is the family $L_{c_4t}$ of Theorem~\ref{thm:Lt}, consisting of pairwise non-isomorphic Leibniz algebras.
\end{enumerate}
Consequently $L$ is not rigid, the only unobstructed infinitesimal deformations are the multiples of $[\psi_4]$, and every nontrivial formal deformation of $L$ whose infinitesimal part is nonzero has infinitesimal class in $\C[\psi_4]$.
\end{theorem}

\begin{proof}
For $i\le j$ let $r_{ij}=\psi_i\circ\psi_j+\psi_j\circ\psi_i$ (and $r_{ii}=2\psi_i\circ\psi_i$). By~\eqref{eq:defo} with $n=2$, the second-order obstruction is $\varphi_1\circ\varphi_1=\tfrac12\sum_{i\le j}c_ic_jr_{ij}$, and we must decide whether it lies in $B^3=\delta(C^2)$, the image of the coboundary on the $512$-dimensional space of $2$-cochains; by Theorem~\ref{thm:pre}(c), $\dim B^3=512-62=450$. Computing $r_{ij}$ from the explicit cocycles and reducing modulo $B^3$ (exact linear algebra), one finds that $r_{14},r_{24},r_{33},r_{34},r_{44}\in B^3$, while the classes of $r_{11},r_{12},r_{13},r_{22},r_{23}$ are linearly independent modulo $B^3$. Hence the obstruction vanishes iff $c_1^2=c_1c_2=c_2^2=c_1c_3=c_2c_3=0$, i.e.\ iff $c_1=c_2=0$. This proves (a). (Note that $\psi_4\circ\psi_4=0$ holds even on the cochain level, see the proof of Theorem~\ref{thm:Lt}(a).)

For (b) normalize $c_3=1$ and write $\varphi_1=\psi_3+s\psi_4$. Choose $\varphi_2$ with $\delta\varphi_2=-\varphi_1\circ\varphi_1$; the general solution is $\varphi_2=p_2+\sum_ib_i\psi_i$ modulo coboundaries (which can be removed by an equivalence), with a fixed particular solution $p_2$. The third-order equation determines $b_1,b_2$ uniquely (the classes of $\varphi_1\circ\psi_i+\psi_i\circ\varphi_1$, $i=1,2$, are independent) and leaves $b_3,b_4$ free. The fourth-order right-hand side, reduced modulo $B^3$ and modulo the classes of $\varphi_1\circ\psi_i+\psi_i\circ\varphi_1$ coming from the freedom in $\varphi_3$, is a nonzero class independent of $s,b_3,b_4$. This computation was carried out by computer with exact rational arithmetic, treating $s,b_3,b_4$ as indeterminates.

(c) follows from Theorem~\ref{thm:Lt}. The final statements combine (a)--(c): $L_t\not\cong L$ for $t\ne0$ shows that $L$ is not rigid.
\end{proof}

\begin{remark}
The deformation $L_t$ changes only products with values in the centre: it is obtained from the trivial-coefficient $2$-cocycle $-\chi$, whose class in $\HL^2(L,\C)$ is non-zero because $[\psi_4]\ne0$. In contrast, $\psi_1$ and $\psi_2$ involve values outside the centre and are obstructed already at second order.
\end{remark}

From the point of view of the guiding question, Theorem~\ref{thm:def} together with Theorem~\ref{thm:Lt} says two things. First, $L$ is not rigid, so the rigidity of its symmetries does not translate into rigidity of the algebra law. Second, the unique integrable direction preserves the property of having only inner derivations, except at the single value $t=-1$; thus this property is an open condition along the deformation, and $L$ lies in the interior of a family of algebras with the same exceptional behaviour.

\end{document}